\documentclass[11pt, reqno]{amsart}
\usepackage{graphicx,psfrag}
\usepackage[psamsfonts]{amssymb}
\usepackage{pdfsync}
\usepackage{amscd}
\usepackage{eufrak}
\usepackage{amsmath}
\usepackage{amsxtra}
\usepackage{mathrsfs}
\usepackage{stmaryrd}
\usepackage[small,nohug,heads=littlevee]{diagrams}
\usepackage{comment}
\usepackage{enumerate}
\usepackage{xfrac}
\usepackage{hyperref}
\diagramstyle[labelstyle=\scriptstyle]
 \usepackage[usenames,dvipsnames]{pstricks}
 \usepackage{pstricks-add}
 \usepackage{pst-grad} 
 \usepackage{pst-plot} 
 \usepackage{pst-eps}
 \usepackage{pst-node}
\usepackage{pst-tree}

\usepackage{mathpple}

\usepackage[width=6.3in, height=8.5in, bottom=1.3in, centering]{geometry}

\theoremstyle{plain}
    \newtheorem{theorem}{Theorem}[section]
    \newtheorem{lemma}[theorem]{Lemma}
    \newtheorem{corollary}[theorem]{Corollary}
    \newtheorem{proposition}[theorem]{Proposition}

\theoremstyle{definition}
    \newtheorem{definition}[theorem]{Definition}

\theoremstyle{remark}
    \newtheorem{remark}[theorem]{Remark}

\numberwithin{equation}{section}

\newcommand{\ZZ}{\mathbb{Z}}

\newcommand{\HH}{\mathbb{H}}

\newcommand{\Econn}{\mathbb{E}}

\newcommand{\Kconn}{\mathbb{K}}

\newcommand{\Nconn}{\mathbb{N}}

\newcommand{\bNconn}{\overline{\mathbb{N}}}

\newcommand{\Edot}{E^{\bullet}}

\newcommand{\kk}{\mathbb{K}}

\newcommand{\complex}{\mathbb{C}}

\newcommand{\A}{\mathcal{A}}

\newcommand{\AOO}{\A^{0,0}}

\newcommand{\AOD}{\A^{0,\bullet}}

\newcommand{\Adot}{\mathcal{A}^{\bullet}}
\newcommand{\PA}{\mathcal{P}_{A}}

\newcommand{\Ho}{\operatorname{Ho}}
\newcommand{\Hom}{\operatorname{Hom}}
\newcommand{\End}{\operatorname{End}}
\newcommand{\Id}{\operatorname{Id}}
\newcommand{\dbar}{\overline{\partial}}

\newcommand{\normal}{N}
\newcommand{\conormal}{N^\vee}

\newcommand{\Dcoh}{\mathcal{D}^b_{coh}}

\newcommand{\Cinf}{$C^\infty$}

\newcommand{\Shat}{\hat{S}}
\newcommand{\ShatV}{\hat{S}(V^\vee)}

\newcommand{\nconn}{\nabla^\bot}

\newcommand{\Dnormal}{\mathfrak{D}}

\newcommand{\upscript}[1]{{\scriptscriptstyle{#1}}}

\newcommand{\Yhat}{{\hat{Y}}}
\newcommand{\Yhatfinite}{{\hat{Y}^{\upscript{(r)}}}}

\newcommand{\Osheaf}{\mathscr{O}}

\newcommand{\Asheaf}{\mathscr{A}}
\newcommand{\Isheaf}{\mathscr{I}}

\newcommand{\formal}{^{\upscript{(\infty)}}}
\newcommand{\Xformal}{{X\formal_Y}}

\newcommand{\dash}{\operatorname{-}}

\newcommand{\Der}{\mathcal{D}er}

\newcommand{\Ldual}{\check{L}}
\newcommand{\Ptilde}{\widetilde{P}_K}
\newcommand{\Kd}{\check{K}}
\newcommand{\disp}{\displaystyle}

\newcommand{\contract}{\llcorner}
\newcommand{\Ktot}{K^\bullet_{Tot}}
\newcommand{\normalbar}{\overline{N}}
\newcommand{\conormalbar}[1]{\normalbar^{{\upscript{\vee} #1}}}
\newcommand{\normalext}{\overline{\mathcal{N}}}
\newcommand{\conormalext}{\normalext^{\upscript{\vee}}}
\newcommand{\Ebar}{\overline{E}}

\newcommand{\Nconnbar}[1]{\bNconn^{\upscript{\vee #1}}}
\newcommand{\Econnbar}{\overline{\Econn}}

\newcommand{\Res}{\operatorname{Res}}
\newcommand{\Rbar}{\overline{R}}
\newcommand{\Alt}{\operatorname{Alt}}
\newcommand{\tr}{\operatorname{tr}}
\newcommand{\treescale}{0.7}
\newcommand{\RHom}{\operatorname{RHom}}
\newcommand{\RHoms}{\mathcal{RH}om}
\newcommand{\Todd}{\operatorname{Td}_\sigma (X/Y)}
\newcommand{\Zplus}{\mathbb{Z}_+}
\newcommand{\Xfirst}{X^{\upscript{(1)}}_Y}
\newcommand{\Xsecond}{X^{\upscript{(2)}}_Y}
\newcommand{\pr}{\operatorname{pr}}

\newarrow{Equal} =====

\title{Todd class via Homotopy Perturbation Theory}

\author{Shilin Yu}
\address{Department of Mathematics, University of Pennsylvania, PA 19104-6395, USA
}
\email{shilinyu@math.upenn.edu}


\thanks{This research was partially supported by the grant DMS1101382  from the National Science Foundation.}


\begin{document}

\begin{abstract}
  We compute the quantized cycle class of a closed embedding of complex manifolds defined by Grivaux using homotopy perturbation theory. In the case of a diagonal embedding, our approach provides a novel perspective of the usual Todd class of a complex manifold.
\end{abstract}

\maketitle


\section{Introduction}
In \cite{GrivauxEuler}, Grivaux proved a conjecture by Kashiwara and Schapira (\cite{KashSch}) concerning Euler classes of coherent  sheaves, which leads to a nice short proof of the Grothendieck-Riemann-Roch (GRR) theorem in Hodge cohomology for arbitrary proper holomorphic maps between complex manifolds  (also see \cite{Mukai1}, \cite{Mukai2}, \cite{ToledoTong4}). Later in \cite{GrivauxHKR}, he generalized Kashiwara's construction and defined the \emph{quantized cycle class} $q_\sigma(X)$ in $\bigoplus_i H^i (X, N^\vee)$ for any closed embedding $i: X \hookrightarrow Y$ of complex manifolds with conormal bundle $N^\vee$ and a splitting $\sigma$ of its first-order formal neighborhood. In the case of the diagonal embedding $X \hookrightarrow X \times X$, $q_\sigma(X)$ is the usual Todd class of $X$. It remains a question how to compute the class $q_\sigma(X)$ in general. 

In this paper, we answer Grivaux's question by analyzing the formal neighborhood $X\formal_Y$ of the embedding. The main tools we will use are the Dolbeault differential algebra (dga) of a formal neighborhood defined by the author in \cite{DolbeaultDGA} and the dg-category of cohesive modules over the Dolbeault dga, which was developed by Block (\cite{Block1}). Our goal is to compute the class $q_\sigma(X)$ (which we will later denote as $\Todd$ and call as the \emph{generalized Todd class} later) explicitly in terms of the differential geometry of the embedding using homotopy perturbation theory (\cite{Perturb}). Even in the case of the diagonal embedding, our method gives the usual Todd class in a novel way which has not appeared in the literature. The Bernoulli numbers emerge naturally from the recursive formulas of the homotopy perturbation theory.

The Hochschild-Kostant-Rosenberg (HKR) isomorphism in algebraic geometry states that the derived tensor product $\Osheaf_X \otimes^{L}_{\Osheaf_{X \times X}} \Osheaf_X$ ($X$ regarded as the diagonal of $X \times X$) is isomorphic in the derived category of $\Osheaf_X$-modules to $\bigoplus_i \Omega^i_X[i]$, where $\Omega^i_X$ is the sheaf of $i$-forms of $X$. Consider the case of a closed general embedding $i: X \hookrightarrow Y$ of complex manifolds of codimension $d$, then in general it is not true that $\Osheaf_X \otimes^L_{\Osheaf_Y} \Osheaf_X$ is isomorphic to its cohomology sheaf $\bigoplus_i \wedge^i N^\vee$. In \cite{ArinkinCaldararu}, however, Arinkin and C{\u{a}}ld{\u{a}}raru showed that such isomorphism exists if and only if the normal bundle $N$ can be extended to a vector bundle $\mathcal{N}$ over the first-order formal neighborhood $\Xfirst$ of $X$ in $Y$. The HKR isomorphism they constructed depends on the choice of the extension $\mathcal{N}$.

One can also consider the dual HKR isomorphism. In other words, we consider the complex $\RHoms_{\Osheaf_Y}(\Osheaf_X,\Osheaf_X)$, which specializes to Hochschild cohomology in the case of a diagonal embedding. Then with the assumption of Arinkin and C{\u{a}}ld{\u{a}}raru on the first-order formal neighborhood, one can find an isomorphism between $\RHoms_{\Osheaf_Y}(\Osheaf_X,\Osheaf_X)$ and $\oplus_i \wedge^i N[-i]$ in $D^b(\Osheaf_X)$ by resolving the first copy of $\Osheaf_X$ as $\Osheaf_Y$-module. However, there is a composition of isomorphisms in $D^b(Y)$ (all functors are derived),
  \begin{equation}\label{eq:HKRdual}
     \RHom_{\Osheaf_Y} (i_* \Osheaf_X, i_*\Osheaf_X) \simeq i_* \RHom_{\Osheaf_X}(\Osheaf_X, i^! i_* \Osheaf_X) \simeq i_* ( i^*i_*\Osheaf_X  \otimes \omega_{X/Y}[-d]) \simeq  \bigoplus_i \wedge^i N[-i], 
  \end{equation}
where the first one is the Grothendieck-Verdier (GV) duality and the second one is again the HKR isomorphism $i_*i^* \Osheaf_X$. The GV isomorphism is in general not $\Osheaf_X$-linear. To fix this, we assume additionally there is a holomorphic retraction $\sigma$ from $\Xfirst$ to $X$ and take the extension $\mathcal{N}= \sigma^* N$ in the theorem of Arinkin and C{\u{a}}ld{\u{a}}raru. Then the isomorphism \eqref{eq:HKRdual} is $\Osheaf_X$-linear. Now we have two isomorphisms between $\RHoms_{\Osheaf_Y}(\Osheaf_X,\Osheaf_X)$ and $\bigoplus_i \wedge^i N[-i]$, which are not the same. Essentially the quantized cycle class $q_\sigma(X)$ measures their difference. As we will see in our main result, Theorem \ref{thm:main}, and \S ~\ref{sec:gamma}, $\Todd$ fully depends on the second-order formal neighborhood $\Xsecond$. We will show that, when $\sigma^*N$ can be extended to a vector bundle $\Xsecond$, $\Todd$ can be represented by a formula similar to that of the usual Todd class. We also describe the obstruction class for the existence of such an extension.

The paper is organized as follows. In \S~\ref{subsec:perturb}, we will review the basic facts about the homotopy perturbation theory. In \S~\ref{subsec:Koszul_complex}, we consider the toy model of the Koszul complex on the formal neighborhood of the origin of a vector space and study its homological algebraic properties. In~\S \ref{subsec:DolbeaultDGA}, we recall the definition and properties of the Dolbeault dga of a formal neighborhood to the extend we need. In \S~\ref{subsec:cohesive}, we review the basics about cohesive modules. In \S~\ref{sec:Koszul}, we globalize the constructions in \S~\ref{subsec:Koszul_complex} and construct Koszul-type resolutions in the category of cohesive modules. In \S~\ref{subsec:Todd}, we prove our main result for the quantized cycle class. Finally, \S~\ref{sec:gamma} is devoted to the explanation of certain cohomological classes in terms of the second-roder formal neighborhood $\Xsecond$.

\section{Preliminaries}

\subsection{Perturbation Lemma}\label{subsec:perturb}

We recall the basic setup of the \emph{Homological Perturbation Theory} from \cite{Perturb}. Let $(A^\bullet,d_A)$ and $(B^\bullet,d_B)$ be two complexes. Consider a contraction
\begin{displaymath}
  F: B^\bullet \to A^\bullet \quad G: A^\bullet \to B^\bullet, \quad H: B^\bullet \to B^{\bullet-1}
\end{displaymath}
which satisfies the usual identities
\begin{equation}\label{eq:contraction}
  FG = 1_A, \quad 1_B - GF = d_B H + H d_B.
\end{equation}
Adjusting the homotopy $H$ if necessary, one can always assume that the following `side conditions' are also satisfied,
\begin{equation}\label{eq:side_condition}
  FH = 0, \quad HH = 0, \quad HG=0.
\end{equation}
Now suppose we are given a new differential $d_B + t$ on $B^\bullet$ such that $tH$ is locally nilpotent (i.e. for any element $n \in B^\bullet$ there is a positive integer $k(n)$ such that $(tH)^{k(n)}(n)=0$). Then the infinite sum
\begin{displaymath}\label{eq:perturb_X}
  X = t - tHt + tHtHt - \cdots
\end{displaymath}
is well-defined. Introduce
\begin{equation}\label{eq:perturb_contraction}
  F_t = F(1-XH), \quad G_t = (1-HX)G, \quad H_t = H - HXH, (d_A)_t = d_A + FXG
\end{equation}

\begin{proposition}[Basic Perturbation Lemma, \cite{Perturb}]\label{prop:perturbation}
  Under the assumptions above, $(F_t,G_t,H_t)$ is a contraction of the complex $(B^\bullet,d_B + t)$ to the complex $(A^\bullet,(d_A)_t)$ which also satisfies the side conditions.
\end{proposition}

\subsection{Koszul complex}\label{subsec:Koszul_complex}

Let $V$ be a vector space over $k=\complex$ of dimension $d$. We want to describe the Koszul resolution of the origin $0 \in V$ in its formal neighborhood $\hat{V}$. The algebra of functions over the single point $0$ is just $k$, while the algebra of functions over $\hat{V}$ is the algebra of formal power series $\Shat(V^\vee) = \prod_{i=0}^{\infty} S^i V^\vee$ the completed free symmetric algebra generated by $V^\vee$. We then have the (completed) Koszul complex $K^\bullet=K^\bullet(V)$
\begin{equation}\label{sq:Koszul_resolution}
  0 \to \Shat(V^\vee) \otimes_\kk \wedge^d V^\vee \xrightarrow{d_K} \cdots \to \Shat(V^\vee) \otimes_\kk \wedge^2 V^\vee \xrightarrow{d_K} \Shat(V^\vee) \otimes_\kk V^\vee \xrightarrow{d_K} \Shat(V^\vee) \to 0
\end{equation}
such that the rightmost $\Shat(V^\vee)$ is of homological degree $0$. Regard $\Shat(V^\vee) \otimes \wedge V^\vee$ as a graded algebra, then the Koszul differential $d_K$ can be defined as the unique $\ShatV$-linear derivation of degree $1$ satisfying 
  \[ d_K( 1 \otimes \overline{v} ) = v \otimes 1, \quad d_K(v \otimes 1) = 0, \quad \forall ~ v \in V^\vee,  \]
where $\overline{v}$ is the same as $v$ as an element in $V^\vee$, yet we distinguish them with different notations so that $v$ is in the symmetric part and has degree $0$ while $\overline{v} \in \wedge^1 V^\vee$ has degree $-1$. We have $d^2_K = 0$.

We also define a derivation $\Ptilde$ on $K^\bullet(V)$ of degree $-1$, determined by 
  \[ \Ptilde(v \otimes 1) = 1 \otimes \overline{v}, \quad \Ptilde(1 \otimes \overline{v}) = 0, \quad \forall ~ v \in V^\vee. \]
Then again $\Ptilde^2 = 0$. Moreover, we have
  \[ [\Ptilde , d_K] = \Ptilde \circ d_K + d_K \circ \Ptilde = (k + l) \Id \quad \text{on} ~ S^l \otimes \wedge^k, \]
which can be checked on the generators. We then define a map $P_K$ on $K^\bullet$ of degree $-1$ by
\begin{equation}
  P_K =
    \begin{cases}
      \disp \frac{1}{k+l} \Ptilde & \text{if} ~ k+l \geq 0  \\
      0 & \text{if} ~ k=l=0
    \end{cases} 
    \quad \text{on} ~ S^l \otimes \wedge^k.
\end{equation}
Note that $P_K$ is no longer a derivation of $K^\bullet$. Neither is it $\ShatV$-linear. Explicitly, we have
\[
  P_K(x \otimes \overline{w}_1  \wedge \cdots \wedge \overline{w}_k) = \frac{1}{k+l}\sum_{i=1}^k (-1)^{i-1} x \cdot w_i \otimes \overline{w}_1 \wedge \cdots \wedge \hat{\overline{w}}_i \wedge \cdots \wedge \overline{w}_k,
\]
for any $x \in S^l V$ and $w_i \in V^\vee$, with $k + l > 0$.

We define two additional maps: the projection $\pi_K: K^\bullet \to \complex$ such that its restriction on $K^0=\Shat(V^\vee) $ is the natural algebra homomorphism $\Shat(V^\vee) \to \complex$ by taking the constant term, and zero elsewhere; the inclusion $i_K$ is the composition $\complex \to K^0 = \Shat(V^\vee) \hookrightarrow K^\bullet$.

\begin{lemma}\label{lemma:Koszul}
  \begin{enumerate}[1)]
    \item
      $P_K^2 = 0$.
    \item
      $\pi_K i_K = \Id_{\complex}$, $d_K P_K + P_K d_K  = \Id_{K^\bullet} - i_K \circ \pi_K$.
    \item
      $(\pi_K, i_K, P_K)$ defines a contraction from $(K^\bullet,d_K)$ to $\complex$.
  \end{enumerate}
\end{lemma}

We define the dual Koszul complex $\Kd^\bullet = \Kd^\bullet(V)$ to be $(K^\bullet)^\vee = \Hom^\bullet_{\Shat(V^\vee)} (K^\bullet,\Shat(V^\vee))$, with the differential
\begin{displaymath}
  d_{\check{K}} (f) =d_{Hom} (f) = (-1)^{i+1} f \circ d_K
\end{displaymath}
for $f \in \check{K}^i = \Hom_{\Shat}(K^{-i},\Shat)$. We have the identification
\begin{equation}\label{eq:Kd_wedge}
  \Kd^i=\Hom_{\Shat}(K^{-i}, K^0)= \ShatV \otimes_\kk \Hom_\kk (\wedge^i V^\vee, \kk) \simeq \ShatV \otimes \wedge^i V,
\end{equation}
where the natural isomorphism $\Hom_\kk (\wedge^i V^\vee, \kk) \simeq \wedge^i V$ is defined by contracting elements of $\wedge^i V$ with $\wedge^i V^\vee$.

\begin{lemma}
 Under the identification \eqref{eq:Kd_wedge}, the differential $d_{\check{K}}$ can be written as  
 \[   
    d_{\Kd} (w_1 \cdots w_l \otimes \overline{v}_1 \wedge \cdots \wedge \overline{v}_k) = -\sum_{i=1}^d w_1 \cdots w_l \cdot \check{e}_i \otimes \overline{e}_i \wedge \overline{v}_1 \wedge \cdots \wedge \overline{v}_k,
  \]
for $w_i \in V^\vee$, $\overline{v}_i \in V$, where $\{ e_i \}$ is a basis of $V$ and $\{ \check{e}_i \}$ is its dual basis in $V^\vee$. 
\end{lemma}

We define a homotopy operator $P_{\Kd}$ of degree $-1$ on $\Kd$ by
\begin{equation}\label{eq:P_Kd}
  \begin{split}
  &P_{\Kd} (w_1 \cdots w_l \otimes \overline{v}_1 \wedge \cdots \wedge \overline{v}_k ) \\
    =  &\frac{1}{l+d-k}  \sum_{\substack{1 \le i \le l \\ 1 \le j \le k}} (-1)^{j} \langle w_i, v_j \rangle w_1 \cdots \hat{w}_i \cdots w_l \otimes \overline{v}_1 \wedge \cdots \wedge \hat{\overline{v}}_j \wedge \cdots \wedge \overline{v}_k,  
   \end{split}
\end{equation}
for $w_i \in V^\vee$, $\overline{v}_i \in V$, whenever $l+d-k > 0$, otherwise $P_{\Kd}=0$. Note that in \eqref{eq:P_Kd} the sign is $(-1)^j$ instead of $(-1)^{j-1}$. Define the cochain map $\pi_{\Kd}: \Kd^\bullet \to \wedge^n V[-d]$ so that its restriction on $\Kd^d = \ShatV \otimes \wedge^d V$ is the natural projection, and zero elsewhere. We also set $i_{\Kd}: \wedge^d V[-d] \to \Kd^n[-d] = \ShatV \otimes \wedge^n V [-d] \subset \Kd^\bullet $ to be the natural inclusion. 

\begin{lemma}\label{lemma:Koszul_dual}
  \begin{enumerate}[1)]
    \item
      $P_{\Kd}^2 = 0$.
    \item
      $\pi_{\Kd} i_{\Kd} = \Id_{\wedge^n V[-d]}$, $d_{\Kd} P_{\Kd} + P_{\Kd} d_{\Kd}  = \Id_{\Kd} - i_{\Kd} \circ \pi_{\Kd}$.
    \item
      $(\pi_{\Kd}, i_{\Kd}, P_{\Kd})$ defines a contraction from $(\Kd^\bullet,d_{\Kd})$ to $\wedge^n V [-d]$.
  \end{enumerate}
\end{lemma}

\begin{proof}
  We have the isomorphism $\check{K}^\bullet \simeq K^\bullet \otimes (\wedge^d V [-d])$ via the natural identification $\wedge^k V^\vee \otimes \wedge^d V \simeq \wedge^{d-k} V$. Under this isomorphism, the quadruple $(d_{\Kd}, P_{\Kd}, \pi_{\Kd}, i_{\Kd})$ coincides with $(d_K \otimes 1, P_K \otimes 1, \pi_K \otimes 1, i_K \otimes 1)$ up to signs. The conclusion now follows from Lemma \ref{lemma:Koszul}.
\end{proof}

We form the dga of $\ShatV$-linear isomorphisms of $K^\bullet$
\[ 
  \End^\bullet_{\ShatV}(K^\bullet) =  \Hom^\bullet_{\ShatV}(K^\bullet,K^\bullet)
\]
with the differential 
\[
    d_{\Hom} (f) = d_K f + (-1)^{|f|+1} f d_K. 
\]
We have the canonical identification 
\begin{equation}\label{eq:Hom_tensor}
   \Hom^\bullet_{\Shat}(K^\bullet,K^\bullet) \simeq K^\bullet \otimes_{\Shat} \Hom^\bullet_{\Shat}(K^\bullet,\ShatV) \simeq K^\bullet \otimes_{\Shat} \Kd^\bullet = \bigoplus_{p-q=\bullet} \Shat \otimes \wedge^p V^\vee \otimes \wedge^q V,
\end{equation}
under which the differential $d_{\Hom}$ can be written as
\[
 d_{\Hom} = d_K \otimes 1 + 1 \otimes d_{\Kd},
\]
where the graded tensors obey the Koszul sign rule. More explicitly, for $f \in \Hom(K^{-i}, K^{-i+k}) = \Shat \otimes \wedge^{i-k} V^\vee \otimes \wedge^i V$,
\begin{equation}\label{eq:d_Hom_tensor}
  d_{\Hom}(f) = d_{K} f + (-1)^{i-k} d_{\Kd} f,
\end{equation}
where $d_{\Kd}$ acts on the $\wedge^\bullet V$-component in the tensor representation \eqref{eq:Hom_tensor}.

We construct two different contractions from $\Hom^\bullet$ to the complex $\wedge^\bullet V$ with zero differential. We first define $i_H: \wedge^\bullet V \to \End^\bullet(K^\bullet)$ by
\[ i_H(v_1 \wedge \cdots \wedge v_k) = \iota_{v_1} \circ \cdots \circ \iota_{v_k} \]
for any $v_1 \wedge \cdots \wedge v_k \in \wedge^k V$, where $\iota_{v}$ is the contraction of $v \in V$ with elements in $\wedge^\bullet V^\vee$. $i_H$ is a homomorphism of graded algebras, where $\wedge^\bullet V$ is endowed with the exterior product, while the multiplication on $\End^\bullet$ is the usual composition of maps. Moreover, $i_H$ identifies $\wedge^\bullet V$ as a graded commutative subalgebra of $\End^\bullet$.

There are two projections $\pi_{T}$, $\pi_{GV}$ from $\End^\bullet$ to $\wedge^\bullet V$: $\pi_{T}$ acting on the `rightmost component' of $\End^k(K^\bullet)$,
\[  
  \Hom_{\Shat}(K^{-i}, K^0)= \ShatV \otimes_k \Hom_k (\wedge^i V^\vee, k) \simeq \ShatV \otimes \wedge^i V
\]
is the restriction to the constant term of $\ShatV$ ($k \geq 0$), and zero on $\Hom(K^{-i}, K^{-j})$ for $j > 0$; $\pi_{GV}$ is defined similarly, but instead of taking the `rightmost component', $\pi_{GV}$ takes the constant term of 
\[  \Hom_{\Shat}(K^{-d},K^{k-d}) = \ShatV \otimes_\kk \Hom_\kk(\wedge^d V^\vee, \wedge^{k-d} V^\vee) \simeq \ShatV \otimes_\kk \wedge^k V, \]
where we use the identification $\Hom_\kk(\wedge^d V^\vee, \wedge^{k-d} V^\vee) \simeq \wedge^k V$ again via contracting $\wedge^k V$ with $\wedge^d V^\vee$.

Next we define a homotopy operator $P_T$ on $\Hom^\bullet$ of degree $-1$ by
\begin{equation}\label{eq:P_T}
  P_T(f) = \sum_{i \geq 0} (-1)^i P_K  (\delta d_{\Kd} P_K)^i (f),
\end{equation}
where $\delta$ is the sign appearing in \eqref{eq:d_Hom_tensor} and we use the tensor representation of $f$ as in \eqref{eq:Hom_tensor}. The second homotopy operator $P_{GV}$ is defined similarly by
\begin{equation}\label{eq:P_GV}
  P_{GV}(f) = \sum_{i \geq 0} (-1)^i \delta P_{\Kd}(d_K \delta P_{\Kd})^i (f).
\end{equation}

\begin{lemma}\label{lemma:Hom_contraction}
  \begin{enumerate}[1)]
    \item
      $\pi_{T} i_H = \pi_{GV} i_H= \Id_{\wedge^\bullet V}$
    \item  
      We have
        \begin{equation}\label{eq:htpy_T}
          d_{\Hom} P_{T} + P_{T} d_{\Hom}  = \Id_{\Hom} - i_H \circ \pi_{T} 
         \end{equation}
      and 
        \begin{equation}\label{eq:htpy_GV}
          d_{\Hom} P_{GV} + P_{GV} d_{\Hom}  = \Id_{\Hom} - i_H \circ \pi_{GV}. 
        \end{equation}
      Hence $(\pi_T, i_H, P_T)$ and $(\pi_{GV}, i_H, P_{GV})$ define two contractions from $\End^\bullet_{\Shat}(K^\bullet)$ to $\wedge^\bullet V$.
    \item
      Moreover, both contractions satisfy the side condtions, i.e.,
        \[ \pi_T P_T = 0, \quad P_T^2 =0, \quad P_T i_H = 0,\]
      and
        \[ \pi_{GV} P_{GV} = 0, \quad P_{GV}^2 =0, \quad P_{GV} i_H = 0. \]
  \end{enumerate}
\end{lemma}

\begin{proof}
  We only check 2) and leave the proofs of other parts to the reader. To show \eqref{eq:htpy_T}, first notice that we have
\begin{equation}
  d_{\Hom} P_T + P_T d_{\Hom} = 1 - r
\end{equation}
where
\begin{equation}
  r = \sum_{i} (-1)^i (P_T \delta d_{\Kd})^i \Res,
\end{equation}
and $\Res: \Hom(K^{-p}, K^{-q}) \to \Hom(K^{-p}, K^{-q})$ is zero except when $q=0$, in which case $\Res$ takes the constant term of $f \in \Hom_{\Shat}(K^{-\bullet},\Shat) = \Shat \otimes \Kd^\bullet$ in the $\Shat$-component (cf., \cite{ToledoTong1}, (8.17), \cite{ToledoTong2}, (1.10)). Then we can check that $r$ is exactly equal to $i_H \circ \pi_T$. The proof of \eqref{eq:htpy_GV} is similar.
\end{proof}

We also consider the subcomplex $\Der^\bullet_{\Shat}(K^\bullet) \subset \End^\bullet_{\Shat}(K^\bullet)$ of $\ShatV$-linear derivations of the graded algebra $K^\bullet$. Since a derivation is determined by its value on the generator $K^{-1}$, we have
\begin{equation}
  \Der^k_{\Shat}(K^\bullet) \simeq \Hom_{\complex}(V^\vee, \Shat(V^\vee) \otimes \wedge^{k+1} V^\vee) = \Shat(V^\vee) \otimes \wedge^{k+1} V^\vee \otimes V.
\end{equation}
\begin{lemma}\label{lemma:der}
  $P_T$ and $P_{GV}$ preserve $\Der^\bullet_{\Shat}(K^\bullet)$. Moreover, they coincide on $\Der^\bullet$,
  \[  P_T (f) = P_{GV}(f) = P_{K}(f), \]
using the tensor representation of $f \in \Der^\bullet$.
\end{lemma}

\subsection{Dolbeault dga of formal neighborhood}\label{subsec:DolbeaultDGA}

Let $i: X \hookrightarrow Y$ be a closed embedding of complex manifolds. The \emph{$r$-th formal neighborhood $\Yhatfinite$} of $X$ in $Y$ ($r \ge 0$) is the ringed space $(X, \Osheaf_{\Yhatfinite})$ whose structure sheaf is
\begin{displaymath}
  \Osheaf_{\Yhatfinite} = \Osheaf_Y / \Isheaf^{r+1}.
\end{displaymath}
The (complete) formal neighborhood $\Yhat$ of the embedding is the ringed space $(X, \Osheaf_{\Yhat})$ where the structure sheaf is given by
\begin{displaymath}
 \Osheaf_{\Yhat} = \varprojlim_{r} \Osheaf_{\Yhatfinite} = \varprojlim_{r} \Osheaf_X / \Isheaf^{r+1}.
\end{displaymath}
We denote by $\hat{i}: X \hookrightarrow \Yhat$ the embedding of $X$ into the formal neighborhood $\Yhat$. 

In \cite{DolbeaultDGA}, we have defined the \emph{Dolbeault dgas} $(\A^\bullet(\Yhatfinite), \dbar)$ and $(\A^\bullet(\Yhat),\dbar)$ for $\Yhatfinite$ and $\Yhat$, respectively (see Definition 2.3, \cite{DolbeaultDGA}). They are the analogues of the usual Dolbeault complex of a complex manifold. For instance, they can be localized to sheaves of dgas $(\Asheaf^\bullet_{\Yhatfinite},\dbar)$ and $(\Asheaf^\bullet_{\Yhat},\dbar)$, respectively, and give soft resolution of the structure sheaves $\Osheaf_{\Yhatfinite}$ and $\Osheaf_{\Yhat}$, respectively (see Prop. 2.8, \cite{DolbeaultDGA}). Locally, $\Asheaf^\bullet_{\Yhat}$ is isomorphic to $\AOD_X(\Shat(N^\vee))$, the Dolbeault complex of the completed symmetric algebra $\Shat(N^\vee)$ of the conormal bundle $N^\vee$ of the embedding, with the usual differential $\dbar$.

Globally, however, we have to deform the differential $\dbar$ on $\AOD_X(\Shat(N^\vee)$ so that it is isomorphic to $\A^\bullet(\Yhat)$.  This is the main content of Theorem 4.5 of \cite{LinfAlgebroid}. We only briefly summarize the result under the special situation considered here. Recall a holomorphic splitting $\sigma$ of the short exact sequence
\begin{equation}\label{seq:normal}
  0 \to TX \to i^*TY \to N \to 0
\end{equation}
is equivalent to a splitting of the first order formal neighborhood $X^{\upscript{(1)}}_Y \simeq \Osheaf_X \oplus N^\vee$, or equivalently, an isomorphism of dgas
\begin{equation}\label{eq:2nd_formal}
 (\A^\bullet(X^{\upscript{(1)}}_Y), \dbar) \simeq (\AOD_X(S^{\leq 1} N^\vee ), \dbar) = (\AOD(X) \oplus \AOD_X(S^{\leq 1} N^\vee ), \dbar), 
\end{equation}
where the $\dbar$-derivation on the second dga is the usual one. If in addtion we fix a section of the jet bundle of flat torsion-free connections of $TY$ along $X$, by Theorem 4.5., \cite{LinfAlgebroid}, there exists an isomorphism of dgas
  \begin{equation}\label{eq:Dolbeault_isom}
    (\A^\bullet(\Yhat), \dbar) \simeq (\AOD_X(\Shat(N^\vee)), \Dnormal) 
  \end{equation}
which extends the isomorphism \eqref{eq:2nd_formal}, such that the derivation $\Dnormal$ is given by
 \begin{equation}\label{eq:Dnormal}
    \Dnormal = \dbar + \sum_{k \geq 2} \widetilde{R}^\bot_k + \sum_{k \geq 1} \widetilde{B}_k \circ \nabla^\bot.
  \end{equation}
where $R^\bot_k  \in \A^{0,1}_X(\Hom(\conormal, S^k \conormal))$, $R^\bot_2$ is a representative of a class 
  \[  [R^\bot_2] \in H^1(X, \Hom(\conormal, S^2 \conormal)) \] 
given by the projection of the Atiyah class $\alpha_{TY}$ of $TY$ to $H^1(X, \Hom(\conormal, S^2 \conormal))$ determined by the splitting $\sigma$, $B_k \in \A^{0,1}_X(\Hom(T^*X, S^k N^\vee))$,  $~\widetilde{~}~$ stands for the derivation of degree $1$ induced by the tensors and symmetrization, and $\nconn$ is a smooth $(1,0)$-connection on the conormal bundle $N^\vee$ along $X$. Since we have a holomorphic splitting $\sigma$, $B_1 = 0$ and $B_2 = R^\top_2$, which is a representative of the $H^1(X, \Hom(T^*X, S^2 \conormal))$-component of $\alpha_{TY}$. For details see Theorem 4.5. and Remark 4.8. of \cite{LinfAlgebroid}.

In the case of the diagonal embedding $X \hookrightarrow X \times X$, we can choose $\sigma$ by identifying the normal bundle with $\pr^*_2 TX$, where $\pr_2$ is the projection of $X \times X$ to its second component. The isomorphism \eqref{eq:Dolbeault_isom} then specializes to (see \cite{Diagonal})
  \begin{equation}\label{eq:Dolbeaut_isom_diag}
    (\A^\bullet(X\formal_{X \times X}), \dbar) \simeq ( \AOD_X(\Shat (T^*X)), \Dnormal ),
  \end{equation}  
where 
  \begin{equation}\label{eq:Dnormal_diag}
     \Dnormal = \dbar + R_{TX} + \cdots
  \end{equation}
up to order $2$. $R_{TX} \in \A^{0,1}_X( \Hom(T^*X, S^2 T^*X) )$ is a representative of the Atiyah class $\alpha_{TX}$ of $TX$.

\subsection{Cohesive modules}\label{subsec:cohesive}

We recall the definition of cohesive modules over a dga from \cite{Block1}.

\begin{definition}[\cite{Block1}]\label{defn:cohesive}
  Let $A=(\Adot, d)$ be a dga. A \emph{cohesive module} $E=(\Edot,\Econn)$ over $A$ is the following data: 
  \begin{enumerate}
    \item
      A bounded $\ZZ$-graded left module $\Edot$ over $\A=\A^0$ which is finitely generated and projective,
    \item
      a $k$-linear $\ZZ$-connection
        \[\Econn : \Adot \otimes_\A \Edot \to \Adot \otimes_\A \Edot,\]
      which is of total degree one and satisfies the Leibniz condition
        \[\Econn(\omega \otimes e) = d \omega \cdot e + (-1)^{|\omega|} \Econn(1 \otimes e) \]
      and the integrability condition
        \[ \Econn \circ \Econn  = 0.\]
      Such a connection is determined by its value on $\Edot$. So we have $\Econn = \Econn^0 + \Econn^1 + \Econn ^2 + \cdots$, where $\Econn^k : \Edot \to \A^k \otimes_\A E^{\bullet -k+1}$ is the $k$th component of $\Econn$. $\Econn^1$ satisfies the Leibniz rule on each $E^n$ while $\Econn^k$ is $\A$-linear for $k \neq 1$.
  \end{enumerate}
      
   Cohesive modules over $A$ defined as above form a dg-category $\PA$, \emph{the perfect categroy of $A$}, such that the morphism set $\PA^{\bullet}(E_1,E_2)$ between any two cohesive modules $E_1=(E^{\bullet}_1, \Econn_1)$ and $E_2=(E^{\bullet}_2, \Econn_2)$ is a complex, of which the $k$th component $\PA^k(E_1, E_2)$ is defined to be
        \[\{\phi : \Adot \otimes_\A E^\bullet_1  \to \Adot \otimes_\A E^\bullet_2 \vert \deg \phi = k \ \mathrm{and} \ \phi(a \cdot e)=(-1)^{k|a|} a \phi(e), \forall a \in \Adot\}.\]
      The differential $d: \PA^{\bullet}(E_1,E_2) \to \PA^{\bullet+1}(E_1,E_2)$ is defined by
        \[d(\phi)(e)=\Econn_2(\phi(e))-(-1)^{\vert\phi\vert}\phi(\Econn_1(e)).\]
      A morphism $\phi \in \PA^k(E_1,E_2)$ is determined by its restriction to $E^\bullet_1$ and we denote the components of $\phi$ by
        \[\phi^j : E^\bullet_1 \to \A^j \otimes_\A E_2^{\bullet+k-j},\]
      which are all $\A$-linear. The composition map
      \begin{displaymath}
        \PA^i(E_2,  E_3) \otimes_k \PA^j(E_1,E_2) \to \PA^{i+j}(E_1,E_3)
      \end{displaymath}
      is defined by componentwise compositions of $\phi$'s.

\end{definition}

Let $A=(\AOD(X),\dbar)$ be the Dolbeault dga of a compact complex manifold $X$. The following theorem of Block (Theorem 4.3., \cite{Block1}) states that the perfect category $\PA$ associated to the Dolbeault dga provides a dg-enhancement of the derived category $\Dcoh(X)$.

\begin{theorem}[Thm 4.3., \cite{Block1}]\label{thm:perfect_Block}
  Let $X$ be a compact complex manifold and $A=(\AOD(X),\dbar)$ its Dolbeault dga. Then the homotopy category $\Ho \PA$ of the dg-category $\PA$ is equivalent to $\Dcoh(X)$, the bounded derived category of complexes of sheaves of $\Osheaf_X$-modules with coherent cohomology.
\end{theorem}

In \cite{DolbeaultDGA} we have generlized Block's result to the case of formal neighborhoods.

\begin{theorem}[Thm 3.16, \cite{DolbeaultDGA}]\label{thm:perfect_Yu}
Let $i:X \hookrightarrow Y$ be a closed embedding of complex manifolds with $X$ compact. Let $A=(\Adot(\Yhat),\dbar)$ be the Dolbeault dga of the formal neighborhood $\Yhat$, then the homotopy category $\Ho \PA$ of the dg-category $\PA$ is equivalent to $\Dcoh(\Yhat)$, the bounded derived category of complexes of sheaves of $\Osheaf_{\Yhat}$-modules with coherent cohomology.
\end{theorem}

Note that even when the submanifold $X$ is noncompact, there is still a fully faithful functor $\alpha: \Ho \PA \to \Dcoh(\Yhat)$ by Lemma 3.17, \cite{DolbeaultDGA}. We construct in \S~\ref{sec:Koszul} a cohesive module $(K^\bullet,\Kconn)$ over $\Yhat$ for the (derived) direct image $\hat{i}_*\Osheaf$ and study its endomorphism complex in the perfect category $\PA$, whose cohomology is quasi-isomorphic to 
  \[  \RHom^\bullet_{\Osheaf_{\Yhat}} (\hat{i}_* \Osheaf_X, \hat{i}_* \Osheaf_X). \]
We do not need the assumption that $X$ is compact. One important observation is that
  \[  \RHom^\bullet_{\Osheaf_{\Yhat}} (\hat{i}_* \Osheaf_X, \hat{i}_* \Osheaf_X) \simeq  \RHom^\bullet_{\Osheaf_Y } (i_* \Osheaf_X, i_* \Osheaf_X). \]

\section{Koszul resolution on formal neighborhood}\label{sec:Koszul}

In this section, we will build a cohesive module $(K^\bullet, \Kconn)$ over the Dolbeault dga $\A^\bullet(\Yhat)$ which corresponds to $\hat{i}_* \Osheaf_X$. We will compute the low order terms of the connection $\Kconn$ using the description of the Dolbeault dga in \S~\ref{subsec:DolbeaultDGA}.  We will study construct contractions from the total complex of the cohesive module to the Dolbeault complex $(\A^{0,\bullet}(X), \dbar)$ of $X$, i.e., the Dolbeault dga $A_X=(\A^{0,\bullet}(X), \dbar)$. For convenience, we sometimes write the Dolbeault dga $\AOD(X)$ as $\A^\bullet_X$ and $\A^\bullet(\Yhat)$ as $\A^\bullet_{\Yhat}$. 


From now on, we assume there exists a splitting $\sigma$ of the normal short exact sequence \eqref{seq:normal} and we fix an isomorphism $ (\A^\bullet(\Yhat), \dbar) \simeq (\AOD_X(\Shat(N^\vee)), \Dnormal)$ with the differential $\Dnormal$ as in \eqref{eq:Dnormal}.The underlying isomorphism of graded algebras can be thought of as a smooth diffeomorphism between $\Yhat$ and the formal neighborhood $X\formal_N$ of $X$ as the zero section inside the total space of its normal bundle $N$, which of course does not identify the holomorphic structures in general. We set $A_N = (\A^\bullet_N, \dbar) = (\AOD_X(\Shat(N^\vee)),\dbar)$ to be the Dolbeault dga of $X\formal_N$, where $\dbar$ is the usual $\dbar$-derivation induced by the holomorphic structure of $N^\vee$. We have an injective homomorphism $\AOD(X) \to \A^\bullet_N$ of dgas, which corresponds the the contraction from total space of $N$ to its zero section.
 
Let $(L^\bullet,d_L) = (\A^{0,\bullet}_X(N), \dbar)$. We globalize the construction in \S~\ref{subsec:Koszul_complex} and adopt the same notations to obtain the Koszul complex
  \[ K^{-i} =  \A^0_N  \otimes_{\A^0_X} \wedge^i_{\A^0_X} \Ldual^0 = \A^{0,0}_X (\Shat(\conormal) \otimes \wedge^{i} \conormal), \quad 0 \leq i \leq d ,\] 
where $d$ is the codimsion of $X$ in $Y$, and $K^{-i}=0$ otherwise. Here $\otimes_{\A^0_X}$ means tensor over the algebra $\A^{0,0}(X)$ and the exterior and symmetric powers are similar. The total complex is then
\[ 
   \Ktot:= \A^\bullet_N \otimes_{\A^0} K^{\bullet}   \simeq \Shat (\Ldual^\bullet)  \otimes_{\A_X} \wedge^\bullet_{\A_X}(\Ldual^\bullet),
\]   
where $\otimes_{\A_X}$ means tensor over $\AOD(X)$. This can also be regarded a bicomplex with bidegree $(i,j)$, where $i$ comes from the exterior power of the conormal bundle, while $j$ comes from the grading of $\Ldual^\bullet$. That is, the $(i,j)$-component is $K^{i,j}_{Tot} =  \A^j_N \otimes_{\A^0_X} \Shat(\Ldual^0) \otimes_{\A^0_X}  \wedge^{-i}_{\A^0_X} \Ldual^0$.  
   
Regard $K^\bullet$ as a graded $\A^0_X$-algebra and so the Koszul differential $d_K: K^{-i} \to K^{-i+1}$ is an $\A^0_X$-linear derivation of degree $1$. Similar we have an operator $P_K$ of degree $-1$. Let $\Kconn^0=d_K$ and let $\Kconn^1 = \dbar$ be the $\dbar$-connection on $S^\bullet \otimes \wedge^\bullet$ induced by that of $\conormal$. Note that $\Kconn^1$ commutes with $P_K$. Set $\widetilde{\Kconn}=\Kconn^0 + \widetilde{\Kconn}^1$. We have $\widetilde{\Kconn}^2=0$, so $(K^\bullet, \widetilde{\Kconn})$ is a cohesive module over $\A^\bullet_N$ and $(\Ktot, \widetilde{\Kconn})$ forms a cochain complex. 

There are also global versions of $\pi_K$ and $i_K$: the projection $\pi_K: \Ktot \to \A^{0,\bullet}(X)$ such that its restriction on $\A^\bullet_N \otimes_{\A^0} K^0  \subset \Ktot$ is the natural map $\A^\bullet_N \to \A^{0,\bullet}(X)$ and zero elsewhere; the inclusion $i_K$ is the composition $\A^{0,\bullet}(X) \to \A^\bullet_N = \A^\bullet_N \otimes_{\A^0_N} K^0 \hookrightarrow \Ktot$. Then we have an analogue of Lemma \ref{lemma:Koszul}.

\begin{lemma}\label{lemma:Koszul_global}
  \begin{enumerate}[1)]
    \item
      $P_K^2 = 0$.
    \item
      $\pi_K i_K = \Id_{\AOD(X)}$, $\Kconn^0 P_K + P_K \Kconn^0 = \Id_{\Ktot} - i_K \circ \pi_K$, $\widetilde{\Kconn}^1 P_K + P_K \widetilde{\Kconn}^1  = 0$.
    \item
      $(\pi_K, i_K, P_K)$ defines a contraction from $(\Ktot, \widetilde{\Kconn})$ to $(\A^{0,\bullet}(X),\dbar)$.
  \end{enumerate}
\end{lemma}

Hence the cohesive module $(K^\bullet, \widetilde{\Kconn})$ resolves the direct image of $\Osheaf_X$ under the zero section embedding $X \hookrightarrow N$. Again like in \S~\ref{subsec:Koszul_complex}, we can define the dual complex $\Kd^\bullet$ and form the dga $\End^\bullet_{\A_N}(\Ktot) = \Hom^\bullet_{\A^\bullet_N}(\Ktot,\Ktot)$, where
\begin{equation}\label{eq:Hom_K}
  \begin{split}
    \Hom^k_{\A_N}(\Ktot,\Ktot)
      &= \bigoplus_{p,q} \Hom_{\A}(K^{-p}, \A^{q}_N \otimes_{\A_N} K^{k-p-q}) \\
      &= \bigoplus_{p,q} \A^{q}_N \otimes_{\A_N} K^{k-p-q} \otimes_{\A_N} \Kd^{p} \\
      &= \bigoplus_{p,q} \A^q_X(\Shat(\conormal) \otimes \wedge^{p+q-k} \conormal \otimes \wedge^{p} \normal).
  \end{split}
\end{equation}
For any $f \in \A^q_X(\Shat(\conormal) \otimes \wedge^{p+q-k} \conormal \otimes \wedge^{p} \normal)$, we have
\begin{equation}\label{eq:d_Hom}
  \begin{split}
    \tilde{d}_{\Hom} (f)
      &= \dbar f + d_K f + (-1)^{k+1} f d_K \\
      &= \dbar f + d_K f + (-1)^{p-k} d_{\Kd} f .
  \end{split}
\end{equation}

There are two different contractions $(\tilde{\pi}_T, \tilde{i}_T = i_H, \tilde{P}_T)$ and $(\tilde{\pi}_{GV}, \tilde{i}_{GV}=i_H, \tilde{P}_{GV})$ defined in the way as in \S~\ref{subsec:Koszul_complex} from $\End^\bullet_{\A_N}(\Ktot)$ to $(\AOD_X(\wedge^\bullet N), \dbar)$ which satisfy same properties as in Lemma \ref{lemma:Hom_contraction}. If we regard $\End^\bullet$ as a (noncanonical) module over the graded algebra $\AOD(X)$, all maps in the triples here are $\AOD(X)$-linear.

Now we want to make $K^\bullet$ into a cohesive module over $\A^\bullet(\Yhat)$. $K^\bullet$ is already a module over $\A^0(\Yhat)$ under the identification $\A^0(\Yhat) \simeq \A^{0,0}_X(\Shat(N^\vee))$. What we should do is to perturb the connection $\widetilde{\Kconn}$ to make it compatible with the differential on $\A^\bullet(\Yhat)$. The complex $(K^\bullet, \Kconn^0)$ is homotopy equivalent of $\A^0(X)$ as $\A^0(\Yhat)$-modules. Hence by Theorem 3.2.7, \cite{Block1}, there exists a $\ZZ$-connection $\Kconn$ on $\Ktot$ with the $0$-th component being $\Kconn^0$. With the explicit homotopy operator $P_K$ in hand, we will be able to describe part of the connection $\Kconn$ in terms of $\mathfrak{D}$ which will be enough for proving our main result.

We now define components of $\Kconn$ recursively. Since $K^\bullet_{tot}$ is generated by $K^0$ and $K^{-1}$ as a graded algebra, we only need to specify the value of $\Kconn$ on these two components. We have to set $\Kconn$ to be the differential on $K^0 = \A^\bullet(\Yhat)$, which is $\Dnormal$ under the identification $\A^\bullet(\Yhat) \simeq \AOD_X(\Shat(N^\vee))$, since $\Kconn$ is a $\dbar$-derivation. To determine the value of $\Kconn$ on $K^{-1}$, we start with $\Kconn^1$ which is a $\dbar$-derivation. Over $K^{-1}$, we have to choose $\Kconn^1$ so that the squares in the following diagram anti-commute:
\begin{diagram}
&K^{-d} &\xrightarrow{\Kconn^0} &K^{-d+1} &\xrightarrow{\Kconn^0} &\cdots &\xrightarrow{\Kconn^0} &K^{-1} &\xrightarrow{\Kconn^0} &\A^0_\Yhat &\to 0 \\
&\dTo_{\Kconn^1}  & &\dTo_{\Kconn^1}   &  &\cdots  &  &\dTo_{\Kconn^1}   &  &\dTo_{\dbar} \\
&\A^1_{\Yhat} \otimes_{\A^0_\Yhat} K^{-d}  &\to &  \A^1_{\Yhat} \otimes_{\A^0_{\Yhat}} K^{-d+1} &\to &\cdots &\to & \A^1_\Yhat \otimes_{\A^0_\Yhat} K^{-1}  &\to &\A^1_\Yhat &\to 0,
\end{diagram}
i.e.,
\begin{equation}\label{eq:Kconn_anticomm}
  \Kconn^0 \Kconn^1 + \Kconn^1 \Kconn^0 = 0.
\end{equation}

For this purpose, we only need to specify the value of $\Kconn^1$ on $\A^0_X(S^0 N^\vee \otimes \wedge^1 N^\vee) \subset K^{-1}$. We set
  \[  \Kconn^1: \A^0_X (S^0 N^\vee \otimes \wedge^1 N^\vee)  \to \A^1_X (\Shat(N^\vee) \otimes \wedge^1 N^\vee) \]
to be
  \begin{equation}\label{eq:Kconn1}
    \Kconn^1 =\dbar + \sum_{k \geq 2} \Rbar^\bot_k + \sum_{k \geq 1} B_k \circ \nabla^\bot. 
  \end{equation} 
This looks similar to \eqref{eq:Dnormal}, yet the tensors with overlines have different interpretations: $\Rbar^\bot_k = P_K (R^\bot_k)$ so that
  \[ \Rbar^\bot_k \in \A^1_X( S^{k-1} N^\vee \otimes \wedge^1 N^\vee \otimes N). \]
In particular, regarded as an element in $\A^1_X(N^\vee \otimes N^\vee \otimes N)$,
\begin{equation}\label{eq:Rbar2}
  \Rbar^\bot_2 = \frac{1}{2} R^\bot_2.
\end{equation}
and we regard
  \[ B_k \circ \nabla^\bot : \A^0_X (S^0 N^\vee \otimes \wedge^1 N^\vee) \to \A^1_X(S^k N^\vee \otimes \wedge^1 N^\vee)  \]
 by composing $B_k$ with the $T^*X$-component produced by $\nabla^\bot$ yet we do not symmetrize the $S^k N^\vee$-component with the $N^\vee$-component produced by $\nabla^\bot$ as contrast to \eqref{eq:Dnormal}. We can check that such defined $\Kconn^1$ satisfies \eqref{eq:Kconn_anticomm}. 

Now suppose we have constructed components $\Kconn^k$ for $k \leq n$ satisfying
\begin{displaymath}
  \sum_{i=0}^{k} \Kconn^i \Kconn^{k-i} = 0, \quad \forall ~ k \leq n.
\end{displaymath}
Let
\begin{displaymath}
  D = - \sum_{i=1}^{n} \Kconn^{i} \Kconn^{n+1-i},
\end{displaymath}
then $D \in \Hom^2_{\A}(\Ktot,\Ktot)$ and $d_{\Hom}(D) = 0$. In fact, $D$ lies in the subcomplex $\Der^\bullet_{\A}(\Ktot) \subset \End^\bullet_\A(\Ktot)$. Set
\begin{equation}\label{eq:Kconn_recursion}
  \Kconn^{n+1} = P_T(D) = P_T(-\sum_i \Kconn^{i} \Kconn^{n+1-i})\in \Der^1_{\A^\bullet(\Yhat)}(\Ktot,\Ktot)
\end{equation}
(cf. Lemma \ref{lemma:der}). Then by 2), Lemma \ref{lemma:Koszul_global},
\begin{displaymath}
  \sum_{i=0}^{n+1} \Kconn^{i} \Kconn^{n-i} = 0.
\end{displaymath}
Since the complex $\Ktot$ is bounded in both degrees, this process finally stops within finite steps and we get a $\ZZ$-connection $\Kconn = \sum_{i \geq 0} \Kconn^i$ satisfying the integrability condition $\Kconn ^2 = 0$. Thus $(K^\bullet, \Kconn)$ is a cohesive module over $\A^\bullet(\Yhat)$.

If we only care about  terms in $\Kconn^1$ of order less or equal than $2$, we can write
  \begin{equation}\label{eq:Kconn_2nd}
   \Kconn^1 =  \dbar + \Rbar^\bot_2 + \Rbar^\bot_3 + R^\top_2 \circ \nabla^\bot + \text{terms of order } \geq 2. 
   \end{equation}
Then 
  \begin{equation}\label{eq:Kconn_sqr}
    \begin{split}
     (\Kconn^1)^2 
       & = \dbar^2 + [\dbar, \Rbar^\bot_2] + \widetilde{R}^\bot_2 \circ \Rbar^\bot_2 + \Rbar^\bot_2 \circ \Rbar^\bot_2 + [\dbar, \Rbar^\bot_3] + R^\top_2 \circ [\dbar, \nabla^\bot] + \text{terms of order } \geq 2  \\
       &= \widetilde{R}^\bot_2 \circ \Rbar^\bot_2 + \Rbar^\bot_2 \circ \Rbar^\bot_2 + \dbar \Rbar^\bot_3 + R^\top_2 \circ R_N + \text{terms of order } \geq 2.
    \end{split}
  \end{equation}
We explain terms in Eq \eqref{eq:Kconn_sqr} as follows. The first term 
  \[ \widetilde{R}^\bot_2 \circ \Rbar^\bot_2 \in \A^{0,2}_X(S^2 N^\vee \otimes \wedge^1 N^\vee \otimes N) \]
is given by contracting the $N$-component of $R^\bot_2 \in \A^{0,1}_X(S^2 N^\vee \otimes N)$ with the $S^1 N^\vee$-component of $\Rbar^\bot_2 \in \A^{0,1}_X(S^1 N^\vee \otimes \wedge^1 N^\vee \otimes N)$.

The second term $\Rbar^\bot_2 \circ \Rbar^\bot_2$ in Eq \eqref{eq:Kconn_sqr} is given by contracting the $N$-component of $\Rbar^\bot_2$ with the $\wedge^1 N^\vee$-component of $\Rbar^\bot_2$ and multiplying the two $S^1 N^\vee$ from both $\Rbar^\bot_2$ to get the $S^2 N^\vee$-component.

In the fourth term, $R_N=[\dbar, \nabla^\bot] \in \A^{0,1}_X(T^*X \otimes \End(N^\vee))$ is a representative of the Atiyah class $\alpha_{N^\vee}$ of $N^\vee$. The tensor $R_2^\top \circ R_N$ defines a cohomology class
  \begin{equation}\label{eq:gamma_class}
    \gamma_\sigma = [R_2^\top \circ R_N] \in H^2(X, (N^\vee)^{\otimes 3} \otimes N).
  \end{equation}
From now on, we assume $\gamma_\sigma = 0$. Note that $\gamma_\sigma$ depends on the splitting $\sigma$. The geometric meaning of $\gamma_\sigma$ will be explained in \ref{sec:gamma}.

We illustrate Eq \eqref{eq:Kconn_sqr} using binary tree diagram. The convention is as follows: a hollow circle node stands for an $N^\vee$ in the symmetric part, a hollow square node stands for an $N^\vee$ in the antisymmetric part, a solid node, either circular or squared, means it carries a $\A^{0,1}_X$-part (in front of $N^\vee$); a hollow circle has degree $0$, while a solid circle has degree $1$. A hollow square has degree $-1$, while a solid square has degree equal to $1 - 1 = 0$; one branch of a tree stands for a copy of $R^\bot_2$ and a tree expresses compositions of a series of copies of $R^\bot_2$. More explicitly, we have

\begin{figure}[h]
   \centering
   \pstree{\Tc{5pt}}
         {
           \Tc{5pt}
           \Tc*{5pt}
         }
   \rput[l](0.5,-1){$=$}
   \hspace{1cm}
   \pstree{\Tc{5pt}}
         {
           \Tc*{5pt}
           \Tc{5pt}
         }
   \rput(1,-1){$= \quad \widetilde{R}^\bot_2$}
   \caption{Tree diagram for $\widetilde{R}^\bot_2$}
\end{figure}

\begin{figure}[h]
   \centering
   \rput[l](0,-1){$\displaystyle \frac{1}{2} ~ \times$}
   \hspace{0.8cm}
   \pstree{\Tf}
         {
           \Tc{5pt}
           \Tf*
         }
   \rput[l](0.5,-1){$\displaystyle = \quad \frac{1}{2} ~ \times$}
   \hspace{2cm}
   \pstree{\Tf}
         {
           \Tc*{5pt}
           \Tf
         }
   \rput(1.5,-1){$\displaystyle  =  \frac{1}{2}R^\bot_2  =  \Rbar^\bot_2$}
   \caption{Tree diagram for $\Rbar^\bot_2$}
   \label{fig:Kconn2}
\end{figure}

We can express low order terms $\widetilde{R}^\bot_2 \circ \Rbar^\bot_2 + \Rbar^\bot_2 \circ \Rbar^\bot_2 $ in the expression \eqref{eq:Kconn_sqr} of $(\Kconn^1)^2$ in terms of tree diagram (the other two terms are $\dbar$-exact) as illustrated in Fig. \ref{fig:Kconn_sqr}.

\begin{figure}[h]\label{eq:tree2}
   \centering
   \rput[l](0,-2){$\displaystyle \widetilde{R}^\bot_2 \circ \Rbar^\bot_2 + \Rbar^\bot_2 \circ \Rbar^\bot_2  ~ = ~ \frac{1}{2} ~\times$}
   \hspace{4.5cm}
   \pstree{\Tf}
         {
           \pstree{\Tc{5pt}}{\Tc{5pt} \Tc*{5pt}}
           \Tf*
         }
   \rput[l](0.5,-2){$\displaystyle - \quad \frac{1}{2} \times \frac{1}{2} \times$}
   \hspace{2.8cm}
   \pstree{\Tf}
         {
           \Tc*{5pt}
           \pstree{\Tf}{\Tc{5pt} \Tf*}
         }          
   \caption{Tree diagram for $(\Kconn^1)^2$}
   \label{fig:Kconn_sqr}
\end{figure}

Now apply $P_T$ to $-(\Kconn^1)^2$. Note that $P_T$ commutes with $\dbar$, so the images of $\dbar$-exact forms under $P_K$ are still $\dbar$-exact and hence we only need to compute $P_K (-\widetilde{R}^\bot_2 \circ \Rbar^\bot_2)$. This amounts to changing the circular leaves in Fig. \ref{fig:Kconn_sqr} to squared leaves (and do the same to the parent nodes) once at a time and then dividing the trees by total number of the leaves. Notice that if both leaves of a branch are squared, the tree corresponds to the zero tensor since squared leaves are anti-symmetric while the $N^\vee$-part of $R^\bot_2$ is symmetric. The computation is illustrated in Fig. \ref{fig:Kconn_sqrD}. In the second equality we simply switched the solid squared leaf on the first level to the right and since it has degree $0$, this process does not generate a minus sign. 

\begin{figure}[h]
   \centering
   \psset{unit=0.8}
   \rput[l](0,-2){$\displaystyle P_K(-\widetilde{R}^\bot_2 \circ \Rbar^\bot_2 -  \Rbar^\bot_2 \circ \Rbar^\bot_2) ~ = ~ \frac{1}{3} \times \frac{1}{2} ~\times$}
   \hspace{6.5cm}
   \psscalebox{\treescale}{
   \pstree{\Tf}
         {
           \pstree{\Tf}{\Tc{5pt} \Tf*}
           \Tf*
         }
   }
   \rput[l](0.5,-2){$\displaystyle - \quad \frac{1}{3} \times \frac{1}{2} \times \frac{1}{2} \times$}
   \hspace{3.5cm}
   \psscalebox{\treescale}{
   \pstree{\Tf}
         {
             \Tf*
           \pstree{\Tf}{\Tc{5pt}  \Tn \Tn  \Tf* \Tn}
         }      
   }    
   
   \vspace{0.5cm}
   \rput[l](2,-2){$\displaystyle   = ~ \frac{1}{6} ~  \times$}
   \hspace{3.2cm}
   \psscalebox{\treescale}{
   \pstree{\Tf}
         {
           \pstree{\Tf}{\Tc{5pt} \Tf*}
           \Tf*
         }
   }
   \rput[l](0.5,-2){$\displaystyle - \quad \frac{1}{12} ~ \times$}
   \hspace{2.2cm}
   \psscalebox{\treescale}{
   \pstree{\Tf}
         {
           \pstree{\Tf}{\Tc{5pt}  \Tf*}
           \Tf*
         }   
   }       

   \vspace{0.5cm}
   
   \hspace{-1cm}
   \rput[l](-2,-2){$\displaystyle   = ~ \frac{1}{12} ~  \times$}
   \psscalebox{\treescale}{
   \pstree{\Tf}
         {
           \pstree{\Tf}{\Tc{5pt} \Tf*}
           \Tf*
         }
   }
    \caption{Tree diagram for $\Kconn^2$}
    \label{fig:Kconn_sqrD}
\end{figure}

We see that even though the constant term of $\Kconn^2$ is zero, its first order term is a binary tree with a nontrivial rational coefficient. The binary tree (without the coefficient) as a tensor is 
  \[  \Alt [ (R^\bot_2)^{\otimes 2} ] \in \A^{0,2}_X(\wedge^2 N^\vee \otimes \End(N^\vee)), \]
 where we now regard $R^\bot_2$ as inside $\A^{0,1}_X(N^\vee \otimes \End(N^\vee))$ and
 \[(R^\bot_2)^{\otimes 2} \in  \A^{0,2}_X( (N^\vee)^{\otimes 2} \otimes \End(N^\vee)) \] 
is obtained by tensoring the $N^\vee$-parts and composing the $\End(N^\vee)$-parts. The operator $\Alt: (N^\vee)^{\otimes 2} \to \wedge^2 N^\vee$ is the anti-symmetrization map $v_1 \otimes v_2 \mapsto v_1 \wedge v_2$. Since a solid squared leaf corresponds to a $\A^{0,1}_X$-form together with a `fermionic' $N^\vee$, hence having degree $0$, the solid squared leaves commute with each other, so no minus sign will arise if we reorder them. The only circular leaf corresponds to the codomain of $\End(N^\vee)$. In general, any such binary tree with only left subtrees which has $k+1$ levels, one circular leaf and $k$ squared leaves represents the tensor
  \[  \Alt \left[ (R^\bot_2)^{\otimes k} \right] \in \A^{0,k}_X(\wedge^k N^\vee \otimes \End(N^\vee)), \]
where
  \[ (R^\bot_2)^{\otimes k} \in \A^{0,k}_X( (N^\vee)^{\otimes k} \otimes \End(N^\vee) ) \]
and 
  \[ \Alt: (N^\vee)^{\otimes k} \to \wedge^k N^\vee, \quad v_1 \otimes \cdots \otimes v_k \mapsto v_1 \wedge \cdots \wedge v_k, \] 
is the anti-symmetrization map. The following theorem shows that the first order term of any $\Kconn^k$ can be described in terms of $\Alt[(R^\bot_2)^{\otimes (k+1)}]$.

\begin{theorem}\label{thm:Kconn}
 Suppose $\gamma_\sigma = 0$. Then for any $k \geq 2$, the component of $\Kconn^k$ in $\A^{0,k}(S^{\leq 1} N^\vee \otimes \wedge^k N^\vee)$ is equal to 
  \begin{equation}\label{eq:Kconn_Bernoulli}
    \frac{B_{k}}{k!} \Alt \left[ (R^\bot_2)^{\otimes k} \right] \quad \text{mod $\dbar$-exact terms},
  \end{equation}
  where $B_k$ is the $k$-th Bernoulli number. In particular, when $k \geq 3$ is odd, $\Kconn^k=0$ on the first order neighborhood. When $k=1$, we have
  \begin{equation}\label{eq:Kconn1}
    \Kconn^1 = \dbar + \widetilde{R}^\bot_2 + \Rbar^\bot_2 = \dbar + \widetilde{R}^\bot_2 + \frac{1}{2} R^\bot_2.
  \end{equation}
  mod terms of order $\geq 2$.
\end{theorem}

\begin{proof}
  We argue by induction on $k$. We have checked the case when $k=1$ and $k=2$. Assume the theorem holds for all $\Kconn^i$ with $i \leq k$, $k \geq 2$, we need to show that Eq. \eqref{eq:Kconn_Bernoulli} holds for $k+1$. By the recursive formula \eqref{eq:Kconn_recursion} for $\Kconn^{k+1}$, we only need to determine each $P_T (- \Kconn^i \Kconn^{k+1-i})$. By Lemma \ref{lemma:der}, it is the same as $P_K (- \Kconn^i \Kconn^{k+1-i})$.
  
 First we assume $2 \leq i \leq k-1$. Any composition of $\Kconn^j$ with a $\dbar$-exact term is again $\dbar$-exact, hence by the induction hypothesis $\Kconn^i \Kconn^{k+1-i}$ mod $\dbar$-exact term is equal to sum of compositions of the two trees by identifying the root of the $\Kconn^i$-tree with any of the squared leaves of the $\Kconn^{k+1-i}$-tree. However, the only term which will contribute to $P_T(-\Kconn^i \Kconn^{k+1-i})$ is given by inserting the $\Kconn^i$-tree to the squared leaf on the lowest level of the $\Kconn^{k+1-i}$-tree as illustrated in Fig. \ref{fig:Kconn_comp}. The minus sign appears since we pass the degree $1$ derviation $\Kconn^i$ through the solid circular leaf of $\Kconn^{k+1-i}$ which is of degree $1$.
 
 \begin{figure}[h]
   \centering
   \psset{unit=0.8}
   \rput[l](-3,-5.3){$\displaystyle \Kconn^i \Kconn^{k+1-i} ~ = - \frac{B_i}{i!} \times \frac{B_{k+1-i}}{(k+1-i)!} ~\times$}
   \hspace{3.5cm}
   \psscalebox{0.7}{
   \pstree{\Tf}{
           \pstree{\Tf}{
             \psset{linestyle=dashed}
             \pstree{\Tf[linestyle=solid]}{  
               \psset{linestyle=solid}            
                \Tc*{5pt} \Tn 
                \pstree{\Tf}{                 
                   \pstree{\Tf}{
                      \psset{linestyle=dashed}
                      \pstree{\Tf[linestyle=solid]}{
                        \psset{linestyle=solid}
                        \Tc{5pt}
                        \Tf*
                      }
                      \Tn
                   }
                   \Tf*
                  }
             }
             \Tn 
           }  
           \Tf*
   }
   \psbrace[ref=lC,braceWidth=0.01, nodesepA=0.2](2,-7.5)(2,0){$\Kconn^{k+1-i}$}
   \psbrace[ref=lC,braceWidth=0.01, nodesepA=0.2](2,-15)(2,-7.5){$\Kconn^{i}$}   
}
   \rput[l](2,-5.3){$+$}
   \rput[l](2.8,-5.3){\parbox{2cm}{other trees which will be killed by $P_K$}}
    \caption{The tree diagram for $\Kconn^i \Kconn^{k+1-i}$}
    \label{fig:Kconn_comp}
\end{figure}

Applying $P_K$ to the tree in Fig. \ref{fig:Kconn_comp} changes the solid circular leaf from $K^{k+1-i}$ to a solid squared leaf, which then can be switched to the right child of its parent node, and multiplies the coefficient by $-1/(k+2)$. Hence we have
\begin{equation}
  P_T(-\Kconn^i \Kconn^{k+1-i}) = - \frac{B_i B_{k+1-i}}{(k+2) i! (k+1-i)!} \Alt \left[ (R^\bot_2)^{\otimes (k+1)}  \right]
\end{equation}
and the tree diagram is illustrated in Fig. \ref{fig:Kconn_P}.

 \begin{figure}[h]
   \centering
   \psset{unit=0.8}
   \rput[l](-3,-5.3){$\displaystyle P_T(-\Kconn^i \Kconn^{k+1-i}) ~ = -\frac{B_i B_{k+1-i}}{(k+2) i! (k+1-i)!} ~\times$}
   \hspace{4.5cm}
   \psscalebox{0.7}{
   \pstree{\Tf}{
           \pstree{\Tf}{
             \psset{linestyle=dashed}
             \pstree{\Tf[linestyle=solid]}{  
               \psset{linestyle=solid}                        
                \pstree{\Tf}{                 
                   \pstree{\Tf}{
                      \psset{linestyle=dashed}
                      \pstree{\Tf[linestyle=solid]}{
                        \psset{linestyle=solid}
                        \Tc{5pt}
                        \Tf*
                      }
                      \Tn
                   }
                   \Tf*
                  }
                \Tf*
             }
             \Tn 
           }  
           \Tf*
   }
   \psbrace[ref=lC,braceWidth=0.01, nodesepA=0.2](2,-7.5)(2,0){$k+1-i$ squared leaves}
   \psbrace[ref=lC,braceWidth=0.01, nodesepA=0.2](2,-15)(2,-7.5){$i$ squared leaves}   
}
    \caption{The tree diagram for $\Kconn^i \Kconn^{k+1-i}$}
    \label{fig:Kconn_P}
\end{figure}

We are left with the cases of $\Kconn^1 \Kconn^k$ and $\Kconn^k \Kconn^1$. Special care needs to be taken of since $\Kconn^1$ is a $\dbar$-derivation. By Eq. \eqref{eq:Kconn1}, we can write
\begin{equation}\label{eq:Kconn1k}
 \begin{split}
    \Kconn^1 \Kconn^k + \Kconn^k \Kconn^1 
    & = [\dbar, \Kconn^k] + \widetilde{R}^\bot_2 \circ \Kconn^k + \Rbar^\bot_2 \circ \Kconn^k + \Kconn^k \circ \Rbar^\bot_2 \\
    & = \widetilde{R}^\bot_2 \circ \Kconn^k + \Rbar^\bot_2 \circ \Kconn^k + \Kconn^k \circ \Rbar^\bot_2
 \end{split}
\end{equation}
as an equation in $\A^{0,k+1}_X(S^2 N^\vee \otimes \wedge^{k} N^\vee \otimes N)$ mod $\dbar$-exact term. The last equality is because $\Kconn^k=\Alt[(R^\bot_2)^{\otimes k}]$ is $\dbar$-closed. We now apply $-P_T$ to Eq. \eqref{eq:Kconn1k}. The computation of $P_T(-\widetilde{R}^\bot_2 \circ \Kconn^k - \Rbar^\bot_2 \circ \Kconn^k)$ is essentially the same as what we have done in Fig. \ref{fig:Kconn_sqr} and Fig. \ref{fig:Kconn_sqrD}, except for that we add branches to the lowest level of a $\Kconn^k$-tree. Note that we can also generate trees by adding $\Rbar^\bot_2$ to other squared leaves of higher levels, but those trees will be eliminated after we apply $P_T$. Hence we get
  \[ P_T(-\widetilde{R}^\bot_2 \circ \Kconn^k - \Rbar^\bot_2 \circ \Kconn^k) 
      = \frac{B_k}{(k+2)k!} \Alt \left[ (R^\bot_2)^{\otimes (k+1)}  \right] . \]
On the other hand, $P_T(-\Kconn^k \circ \Rbar^\bot_2)$ can be computed using exactly the same trick as used in Fig. \ref{fig:Kconn_comp} and Fig. \ref{fig:Kconn_P}, and we get
  \[  P_T(-\Kconn^k \circ \Rbar^\bot_2) = - \frac{B_k}{(k+2)k!} \Alt \left[ (R^\bot_2)^{\otimes (k+1)} \right]. \]
We see that the coefficients cancel out and thus
 \[  P_T(-\Kconn^1 \Kconn^k - \Kconn^k \Kconn^1) = 0 \in \A^{0,k+1}_X(S^2 N^\vee \otimes \wedge^{k} N^\vee \otimes N) \] 
up to $\dbar$-closed term.

In summary, we now obtain
\begin{equation}
  P_T \left(  - \sum_{i=1}^{k} \Kconn^i \Kconn^{k+1-i}  \right) = A_{k+1} \cdot \Alt \left[ (R^\bot_2)^{\otimes (k+1)} \right]
\end{equation}
where 
\begin{equation}\label{eq:A_recursion}
  A_{k+1} = - \frac{1}{k+2} \sum_{i=2}^{k-1} \frac{B_i B_{k+1-i}}{ i! (k+1-i)!}.
\end{equation}
Finally, to show $A_{k+1} = B_{k+1}/{(k+1)!}$, we first notice that by letting $k=2$ in Eq. \eqref{eq:A_recursion} we get $A_3=0=B_3 / 3!$ and hence for any odd value of $k+1$, $A_{k+1} =0 = B_{k+1}/(k+1)!$ since $B_n = 0$ for $n$ odd except for $n=1$.  For even $k+1$, we have the well-known recursive relation of Bernoulli numbers recorded in Lemma \ref{lemma:Bernoulli} below (see, e.g., \cite{Zeta}).

\end{proof}

\begin{lemma}\label{lemma:Bernoulli}
The Bernoulli numbers satisfy the recursive relation (for $n \geq 2$)
\begin{equation}
   \sum_{i=1}^{n-1} {2n \choose 2i} B_{2i} B_{2n-2i} = - (2n+1) B_{2n}.
\end{equation}
\end{lemma}

\section{Generalized Todd class}\label{subsec:Todd}
We have two differentials, $\widetilde{\Kconn}= \Kconn^0 + \dbar$ and $\Kconn$ on $\Ktot$. We set 
  \[ t= \Kconn - \widetilde{\Kconn} = (\Kconn^1-\dbar) + \sum_{k \geq 2} \Kconn^k. \]
By Theorem \ref{thm:Kconn}, when the class $\gamma_\sigma = 0$, up to first order we have
\begin{equation}\label{eq:t} 
  t = \frac{1}{2} R^\bot_2 +  \sum_{k \geq 2} \frac{B_k}{k!} \Alt \left[ (R^\bot_2)^{\otimes k} \right]  =   \sum_{k \geq 1} (-1)^k \frac{B_k}{k!} \Alt \left[ (R^\bot_2)^{\otimes k} \right] \in \bigoplus_{j \geq 1} \A^{0, j}_X(S^1 N^\vee \otimes \wedge^j N^\vee \otimes N),
\end{equation}  
mod $\dbar$-exact terms, which acts an $\AOD_X(\Shat(N^\vee))$-linear derivation on $\Ktot$ of degree $1$ (note that $B_1=1/2$ and $B_{2n+1}=0$ for $n \geq 1$). $\widetilde{\Kconn}$ and $\Kconn$ induce two differentials $\tilde{d}_{\Hom}$ and $d_{\Hom}$ respectively. Denote their difference by
 \[ T = d_{\Hom} - \tilde{d}_{\Hom} = [t, \dash].  \]
We now apply the Homological Perturbation Theory in \S \ref{subsec:perturb} to $(\End^\bullet_{\A_N}(\Ktot), \tilde{d}_{\Hom})$ and its two sets of contractions $(\tilde{\pi}_T, \tilde{i}_T, \tilde{P}_T)$ and $(\tilde{\pi}_{GV}, \tilde{i}_{GV}, \tilde{P}_{GV})$ to $(\AOD_X(\wedge^\bullet N), \dbar)$ with the perturbation $T$ of $d_{\Hom}$ to obtain the the endomorphism complex $(\End^\bullet_{\A_\Yhat}(K^\bullet_{Tot}), d_{\Hom})$ of the cohesive module $(K^\bullet, \Kconn)$ over $\A^\bullet(\Yhat)$. Since the underlying graded module is unchanged, we will still denote it by $\End^\bullet=\End^\bullet_{\A}(K^\bullet_{Tot})$ and only write the differential differently. We denote the perturbed contractions determined by \eqref{eq:perturb_contraction} as $(\pi_T, i_T, P_T)$ and $(\pi_{GV}, i_{GV}, P_{GV})$ respectively and denote the corresponding perturbed differentials on $\AOD_X(\wedge^\bullet N)$ by $\Nconn_T$ and $\Nconn_{GV}$. We have
\begin{equation}\label{eq:perturb_T}
  \pi_T = \tilde{\pi}_T (1 - Z_T \tilde{P}_T), \quad i_{T} = (1 - \tilde{P}_T Z_T) \tilde{i}_T, \quad \tilde{P}_{T} = \tilde{P}_T - \tilde{P}_T Z_T \tilde{P}_T, \quad \Nconn_T = \dbar + \tilde{\pi}_T Z_T \tilde{i}_T,
\end{equation}
where
\begin{equation}\label{eq:Z_T}
  Z_T = \sum_{k \geq 0} (-1)^k T  (\tilde{P}_T T)^k = \sum_{k \geq 0} (-1)^k (T \tilde{P}_T)^k T.
\end{equation}
Similarly,
\begin{equation}\label{eq:perturb_GV}
\begin{gathered}
  \pi_{GV} = \tilde{\pi}_{GV} (1 - Z_{GV} \tilde{P}_{GV}), \quad i_{GV} = (1 - \tilde{P}_{GV} Z_{GV}) \tilde{i}_{GV}, \quad P_{GV} = \tilde{P}_{GV} - \tilde{P}_{GV} Z_{GV} \tilde{P}_{GV}, \\ \Nconn_{GV} = \dbar + \tilde{\pi}_{GV} Z_{GV} \tilde{i}_{GV},
\end{gathered}
\end{equation}
where
\begin{equation}\label{eq:Z_GV}
  Z_{GV} = \sum_{k \geq 0} (-1)^k T  (\tilde{P}_{GV} T)^k = \sum_{k \geq 0} (-1)^k (T \tilde{P}_{GV})^k T.
\end{equation}

All maps defined above are $\AOD(X)$-linear. Note that it is no longer true that $i_T = i_{GV}$. We get two cohesive modules $(\A^{0,0}_X(\wedge^\bullet N), \Nconn_T)$ and $(\A^{0,0}_X(\wedge^\bullet N), \Nconn_{GV})$ over $(\AOD(X),\dbar)$, whose total complexes are quasi-isomorphic to $(i^*i_*\Osheaf_X)^\vee$ and $i^!i_* \Osheaf_X$ respectively in $D^b_{coh}(X)$. The quasi-isomorphism 
  \[  \pi_T: (\End^\bullet_{\A}(\Ktot), d_{\Hom}) \simeq (\AOD_X(\wedge^\bullet N), \Nconn_T) \]
realizes the adjunction $\hat{i}^* \dashv \hat{i}_*$,
\begin{equation}
  \RHom^\bullet_{D^b(\Yhat)} (\hat{i}_* \Osheaf_X, \hat{i}_* \Osheaf_X) \simeq \RHom^\bullet_{D^b(X)} (\hat{i}^*\hat{i}_*\Osheaf_X, \Osheaf_X).
\end{equation}
The quasi-isomorphism
  \[  \pi_{GV}: (\End^\bullet_{\A}(\Ktot), d_{\Hom}) \simeq (\AOD_X(\wedge^\bullet N), \Nconn_{GV}) \]
realizes the Grothendieck-Verdier duality for $\hat{i}: X \hookrightarrow \Yhat$,
\begin{equation}
  \RHom^\bullet_{D^b(\Yhat)} (\hat{i}_* \Osheaf_X, \hat{i}_* \Osheaf_X) \simeq \RHom^\bullet_{D^b(X)} (\Osheaf_X, \hat{i}^! \hat{i}_* \Osheaf_X).
\end{equation}

\begin{lemma}
  \begin{enumerate}
    \item
      $\pi_T = \tilde{\pi}_T$, $\pi_{GV} = \tilde{\pi}_{GV}$.
    \item 
      $\Nconn_T = \Nconn_{GV} = \dbar$.
  \end{enumerate}
\end{lemma}
\begin{proof}
  We only need to observe that $\pi_T T = \pi_{GV} T = 0$ and hence $\pi_T Z_T = \pi_{GV} Z_{GV} = 0$, since $\pi_T$ and $\pi_{GV}$ map any $S^{\geq 1}N^\vee$ part to zero.
\end{proof}

We want to study the $\AOD(X)$-linear endomorphism $q_\sigma =\pi_T \circ i_{GV}=\tilde{\pi}_T \circ i_{GV}$ of $(\AOD_X(\wedge^\bullet N), \dbar)$. By Eq. \eqref{eq:perturb_T} -- \eqref{eq:Z_GV}, $q_\sigma$ only depends on the parts of $t$ of order less or equal than one. Since $q_\sigma$ is $\AOD(X)$-linear, we only need to specify its value on $\A^{0,0}_X (\wedge^\bullet N)$.

We set
  \begin{equation}
    \Todd = \det \left( \frac{R^\bot_2}{1 - \exp(-R^\bot_2)} \right) = \det \left(\sum_{n=0}^\infty (-1)^n \frac{B_n}{n!} \Alt \left[ (R^{\bot}_2 )^{\otimes n} \right] \right), 
  \end{equation}
which is regarded as a cohomology class in $\bigoplus_{j \geq 0} H^j (X, \wedge^j N^\vee)$. By $\det$ we mean taking determinant of the $\End(N^\vee)$-component. Since $\det( - ) = \exp( \tr \ln( - ))$, we can rewrite 
  \begin{equation}
    \Todd = \exp \left( \tr \ln \left( \frac{R^\bot_2}{1 - \exp(-R^\bot_2)} \right) \right),
  \end{equation}
and by the power series expansion
  \begin{equation}
    \ln \left( \frac{x}{1- e^{-x}} \right) = - \sum_{n=1}^\infty  \frac{B_n}{n! n} x^n,
  \end{equation}
we have
  \begin{equation}
    \Todd = \exp \left( \sum_{n=1}^\infty \frac{1}{n} \rho_n  \right),
  \end{equation}
where 
   \begin{equation}
     \rho_n =  - \frac{B_n}{n!} \tr \left( \Alt \left[ (R^{\bot}_2 )^{\otimes n} \right] \right) \in \A^{0,n}_X(\wedge^n N^\vee).
   \end{equation}

\begin{theorem}\label{thm:main}
  For any closed embedding $X \hookrightarrow Y$ with a chosen splitting $\sigma$ of the first order formal neighborhood such that $\gamma_{\sigma}=0$, the restriction of $q_\sigma$ on $H^\bullet_{\dbar}(X, \wedge^d N) \subset H^\bullet_{\dbar}(X, \wedge^\bullet N)$ is given by contraction with the generalized Todd class $\Todd$.
\end{theorem}

In the case of the diagonal embedding $X \hookrightarrow X \times X$, Theorem \ref{thm:main} specializes to the following corollary by \eqref{eq:Dolbeaut_isom_diag} and \eqref{eq:Dnormal_diag}.

\begin{corollary}
  For the diagonal embedding $\Delta: X \hookrightarrow X \times X$, we choose the splitting $\sigma$ to be the one induced by identifying the normal bundle with $\pr^*_2 TX$, where $\pr_2$ is the projection of $X \times X$ to its second component. Then the restriction of $q_\sigma$ on $H^\bullet_{\dbar}(X, \wedge^d TX) \subset H^\bullet_{\dbar}(X, \wedge^\bullet TX)$ is given by contraction with the usual Todd class of $TX$.
\end{corollary}

\begin{proof}[Proof of Theorem \ref{thm:main}]
  By Eq. \eqref{eq:perturb_T} and \eqref{eq:Z_T},
    \begin{equation}
      i_{GV} = \sum_{k \geq 0} (- \tilde{P}_{GV} T)^k \tilde{i}_{GV},
    \end{equation}

  Since $\tilde{P}_T$ reduces the order of $\Shat$-part by one while $T$ has no constant term, and $\pi_T = \tilde{\pi}_T$ only depends on the $S^0$-terms, we only need the $S^{\leq 1}$-part of $T$ and hence of $t$ (Eq. \eqref{eq:t}). 
    
  
We observe that $\pi_T (- \tilde{P}_{GV} T)^k $ ($\forall ~ k \geq 0$) can only be nonzero on the subspace $M=\AOD_X(\wedge^\bullet N) \subset \End^\bullet_{\A}(\Ktot)$. In other words, we can write
  \begin{equation}\label{eq:q_new}
    q_\sigma = \pi_T i_{GV} =  \sum_{k \geq 0} (- \pi_T \tilde{P}_{GV} T  \tilde{i}_T)^k,
  \end{equation}
since $\tilde{i}_T \pi_{T}$ acts as identity on $M$. For any $\eta \in \A^{0,m}_X(\wedge^l N)$ of total degree $|\eta|=m-l$, we have 
  \begin{equation}\label{eq:pigti}
    \begin{split}
    - \pi_T \tilde{P}_{GV} T  \tilde{i}_T(\eta) 
    &= -\pi_T \tilde{P}_{GV} [t, \tilde{i}_T(\eta)] \\
    &= - \pi_T \tilde{P}_{GV} \left[ t \circ (\tilde{i}_T(\eta)) - (-1)^{|\eta|} \tilde{i}_T(\eta) \circ t \right] \\
    &= (-1)^{|\eta|} \pi_T \tilde{P}_{GV} [\tilde{i}_T(\eta) \circ t] \\
    &= (-1)^{|\eta|}\pi_T \tilde{P}_{GV} (\tilde{i}_T(\eta) \circ t) \\
     &=(-1)^{|\eta|} \pi_T \tilde{P}_{GV} ( t \lrcorner \eta) \\
     & =  (-1)^{|\eta|} \sum_{j \geq 1}   (-1)^j \frac{B_j}{j!} \cdot (-1)^{|\eta|-j} P_{\Kd} \left( \Alt \left[ (R^\bot_2)^{\otimes k} \right] \lrcorner \eta \right) \\ 
     &= - \sum_{j \geq 1} \frac{B_j}{j! (d-l+j)}  \tr \left( \Alt \left[ (R^{\bot}_2 )^{\otimes j} \right] \right) \lrcorner \eta  \\
     &= \sum _{j \geq 1} \frac{1}{(d-l+j)}  \rho_j \lrcorner \eta,
   \end{split}
 \end{equation} 
where minus sign in the last line and the fraction $1/(d-l+j)$ are due to \eqref{eq:P_Kd}. The way the trace operator arises in the last second equality can be visualized as contracting the top root square with only circle leaf at bottom in the tree diagram as in Fig. \ref{fig:Kconn_P}.Plug Eq. \eqref{eq:pigti} into Eq. \eqref{eq:q_new}, we obtain, for any $\eta \in \AOD_X(\wedge^d N)$
  \begin{equation}
    \begin{split}
      q_{\sigma} (\eta) 
       &= \sum_{k \ge 0}\sum_{ j_1, \cdots, j_k \ge 1 }  \prod_{p=1}^{k} \frac{1}{ \sum_{q=1}^{p} j_q}  (\rho_{j_k} \wedge \cdots \wedge \rho_{j_1}) \lrcorner \eta.  \\
    \end{split}    
  \end{equation}
Then $q_{\sigma}(\eta) = \Todd \lrcorner \eta$ follows from Lemma \ref{lemma:frac} below.
 
\end{proof}

\begin{definition}
For any $k$-partition $l=n_1 + n_2 + \cdots + n_k$ of a postitive integer $l$ into $k$ positive integers ($k \ge 1$), let $C(l,k;n_1, \cdots, n_k)$ be the set of all \emph{$k$-compositions}  of $l$ defining the given partition, i.e., all ordered $k$-tuples of positive integers $\mathbf{c}=(c_1, \cdots, c_k)$ which are permutations of $(n_1, \cdots, n_k)$. In other words, a composition is an ordered sequence of positive integers which sum up to $l$, while for partition the order of the summands is ignored. We also denote by $C(l,k)$ the set of all $k$-compositions of $l$.

We set $\hat{C}(k-1;n_1, \cdots, n_{k})$ to be the set of all $(k-1)$-tuple $\hat{c}$ obtained by erasing one of the coordinates of some $c \in C(l,k;n_1, \cdots, n_k)$. Note that if two summands in a $k$-composition are the same, erase either of them might give the same $(k-1)$-composition.

For a $k$-tuple $\mathbf{c}=(c_1, \cdots c_k)$ of positive integers, we set
  \[ R(\mathbf{c}) = \prod_{i=1}^k \frac{1}{c_i}, \quad RS(\mathbf{c}) = \prod_{i=1}^{k} \frac{1}{\sum_{j=1}^i c_j }.  \]
  
\end{definition}

\begin{lemma}\label{lemma:frac}
  Given any partition $l=n_1 + n_2 + \cdots + n_k$, we have the equality
    \begin{equation}
      \sum_{\mathbf{c} \in C(l,k;n_1, \cdots, n_k)} \prod_{i=1}^{k} \frac{1}{\sum_{j=1}^i c_j } = \frac{1}{k!}  \sum_{\mathbf{c} \in C(l,k;n_1, \cdots, n_k)} \prod_{i=1}^{k} \frac{1}{c_i}.
    \end{equation}
\end{lemma}

\begin{proof}
  We argue by induction on $k$. The $k=1$ case is trivial. Suppose the equality holds for for some $k \in \Zplus$ and any $k$-partition of any $l \in \Zplus$, we want to prove it for $k+1$. Suppose $l=n_1 + \cdots + n_{k+1}$ is a $(k+1)$-partition of $l$. Then for any $\mathbf{c}=(c_1, \cdots, c_{k+1}) \in C(l,k+1 ; n_1, \cdots, n_{k+1})$,
  \begin{equation}
    \begin{split}
      \prod_{i=1}^{k+1} \frac{1}{c_i} = \frac{1}{l} \cdot \frac{c_1 + c_2 + \cdots + c_{k+1}}{c_1 c_2 \cdots c_{k+1}} =  \frac{1}{l} \sum_{p=1}^{k+1} \frac{1}{c_1 \cdots \hat{c}_p \cdots c_{k+1}}, 
    \end{split}
  \end{equation}
  hence
    \begin{equation}
      \begin{split}
         & \frac{1}{(k+1)!}  \sum_{\mathbf{c} \in C(l,k+1;n_1, \cdots, n_{k+1})} \prod_{i=1}^{k+1} \frac{1}{c_i}  \\
      = & \frac{1}{(k+1)!}  \sum_{\mathbf{c} \in C(l,k+1;n_1, \cdots, n_{k+1})}  \frac{1}{l}\sum_{p=1}^{k+1} \frac{1}{c_1 \cdots \hat{c}_p \cdots c_{k+1}}   \\
      = & \frac{1}{l} \cdot \frac{1}{(k+1)!} \cdot (k+1) \sum_{\mathbf{\hat{c}} \in \hat{C}( k;n_1, \cdots, n_{k+1})} R(\hat{\mathbf{c}}),  \\
      = & \frac{1}{l} \sum_{\mathbf{\hat{c}} \in \hat{C}( k;n_1, \cdots, n_{k+1})}RS(\hat{\mathbf{c}}) \\
      = & \sum_{\mathbf{c} \in C(l,k+1;n_1, \cdots, n_{k+1})} RS(\mathbf{c})
       \end{split}
     \end{equation}
  where the factor $k+1$ in the third line appears because any erased $n_i$ in the partition can put in $k+1$-different positions to recover a $(k+1)$-composition. The last second equality is by induction hypothesis for $k$-compositions.
\end{proof}

\section{Geometric meaning of the classes $\gamma_\sigma$ and $\Todd$}\label{sec:gamma}

In this section we give interpretations of $\gamma_\sigma$ and $\Todd$ as certain obstruction classes. Recall we have a closed embedding $i: X \hookrightarrow Y$ and a holomorphic splitting $\sigma$ of the short exact sequence \eqref{seq:normal}. $\sigma$ is equivalent to a retraction $p_\sigma: X^{\upscript{(1)}}_Y \to X$ from the first-order neighborhood $X^{\upscript{(1)}}_Y$ of $X$ to $X$, such that $p_\sigma \circ i = \Id_X$. Hence we have the pullback vector bundle  $p_\sigma^* N^\vee$ over $X^{\upscript{(1)}}_Y$ (i.e., a locally free sheaf over $\Osheaf_{X^{\upscript (1)}_Y}$ on $X$), which is an extension of the conormal bundle $N^\vee$, i.e., $i^*(p^*_\sigma N^\vee) \simeq N^\vee$.

\begin{proposition}\label{prop:gamma}
      $p_\sigma^* N^\vee$ can be extended to a holomorphic vector bundle $\overline{N}^\vee$ over the second-order formal neighborhood $X^{\upscript{(2)}}_Y$, if and only if the class $\gamma_\sigma$ vanishes.
\end{proposition}     

\begin{proof}
  With the chosen $\sigma$, the $\dbar$-derivation on the Dolbeault dga $\A^\bullet(X^{\upscript{(2)}}_Y) \simeq \AOD_X(S^{\leq 2} N^\vee )$ of the second order formal neighborhood $X^{\upscript{(2)}}_Y$ is of the form
   \begin{equation}\label{eq:Dnormal_2nd}
      \Dnormal^{(2)} = \dbar + \widetilde{R}^\bot_2 + \widetilde{R}^\top_2 \circ \partial,   
   \end{equation}
where $\partial$ is the $(1,0)$-part of the de Rham differential on $X$. The Dolbeault complex of $p_\sigma^*N^\vee$ over $X^{\upscript{(1)}}_Y$ is isomorphic to $\AOD_X(S^{\leq 1} N^\vee \otimes N^\vee)$ with the usual $\dbar$-connection induced by the holomorphic structure of $N^\vee$. Any extension $\overline{N}^\vee$ of $p_\sigma^*N^\vee$ over $X^{\upscript{(2)}}_Y$ is equvalent to a $\dbar$-connection on the complex $\Adot(\overline{N}^\vee) = \AOD_X(S^{\leq 2} N^\vee \otimes N^\vee)$ compatible with $\Dnormal^{(2)}$ of the form
  \begin{equation}\label{eq:DNvee}
    D_{\overline{N}^\vee} = \dbar + R^\top_2 \circ \nabla^\bot + C,
  \end{equation}
where $C \in \A^{0,1}_X(\Hom(N^\vee, S^2 N^\vee \otimes N^\vee)) = \A^{0,1}_X(S^2 N^\vee \otimes \End(N^\vee))$ (cf. \eqref{eq:Kconn_2nd}). The integrability condtion $( D_{\overline{N}^\vee} )^2 = 0$ is hence equivalent to the equality
  \[  R^\top_2 \circ R_N + \dbar C = 0, \] 
i.e., $\gamma_\sigma = [ R^\top_2 \circ R_N ] = 0 $.  
  
\end{proof}

The following result gives a geometric interpretation of the class $[R^\bot_2] \in H^1(X, S^2 N^\vee \otimes N)$ and strengthens Theorem 1.3 of \cite{GrivauxHKR}.
     
\begin{proposition}
      Suppose that $\gamma_\sigma$ vanishes. Then the class $[R^\bot_2]$ vanishes if and only if for any vector bundle $\overline{N}^\vee$ over $X^{\upscript{(2)}}_Y$ extending $p^*_\sigma N^\vee$ over $X^{\upscript{(1)}}_Y$, its dual $\overline{N}$ admits a holomorphic section vanishing exactly on $X$ and being transverse to the zero section. In this case $\Todd=1$.
\end{proposition}

\begin{proof}
  The Dolbeaut complex $\Adot(\overline{N})$ of $\overline{N}$ over $X^{\upscript{(2)}}_Y$ can be identified with the complex 
  \[  \AOD_X( S^{\leq 2} N^\vee \otimes N) \simeq \Hom (\Adot(\overline{N}^\vee) , \Adot( X^{\upscript{(2)}}_Y ) ). \] 
  By Eq. \eqref{eq:Dnormal_2nd} and \eqref{eq:DNvee} in the proof of Prop. \ref{prop:gamma}, the corresponding $\dbar$-connection of $\overline{N}$ can be written as
  \begin{equation}
    D_{\overline{N}} (\phi) = \Dnormal^{(2)} \circ \phi - (-1)^{|\phi|} \phi \circ D_{\overline{N}^\vee}
  \end{equation}
for $\phi \in  \Hom (\Adot(\overline{N}^\vee) , \Adot( X^{\upscript{(2)}}_Y ) )$. Hence
  \begin{equation}\label{eq:DN}
    D_{\overline{N}} = \dbar + \widetilde{R}^\bot_2 + R^\top_2 \circ \nabla^\bot \pm C,
  \end{equation}
where $\widetilde{R}^\top_2$ acts on the $S^{\leq 2} N^\vee$-part of $\Adot(\overline{N}) =  \AOD_X( S^{\leq 2} N^\vee \otimes N)$, $C$ is regarded as an element of $\A^{0,1}_X(S^2 N^\vee \otimes \End(N))$  we also use the same notion for the connection on $N$ induced by $\nabla^\bot$ on $N^\vee$.
 
Now given a section $v \in \A^{0,0}(\overline{N}) =  \A^{0,0}_X( S^{\leq 2} N^\vee \otimes N)$, we can write it as $v= v_0 + v_1 + v_2$, such that $v_i \in \A^{0,0}_X(S^i N^\vee \otimes N)$, $i=0,1,2$. $v$ vanishes over $X$ if and only if its component $v_0$ in $\A^{0,0}_X(S^0 N^\vee \otimes N)$ is zero. It is transverse to $X$ if and only if $v_1 \in \A^{0,0}_X(S^1 N^\vee \otimes N) = \A^{0,0}_X(\End(N^\vee))$ is invertible. Such $v$ is holomorphic, i.e., $D_{\overline{N}} (v) = 0$, if and only if $\dbar v_1 = 0$ and $ R^\bot_2 \circ v_1 + \dbar v_2= 0$ by \eqref{eq:DN}, which is equivalent to $R^\bot_2 = - (\dbar v_2) \circ v^{-1}_1 = - \dbar ( v_2 \circ v^{-1}_1)$ since $\dbar v_1 = 0$. So we have shown that if such $v$ exists, $[R^\bot_2] = 0$. 

Conversely, assume $[R^\bot_2]=0$, it suffices to set $v_0=0$, $v_1 = \Id \in \A^{0,0}_X(\End(N^\vee)) \subset \A^{0,0}_X( S^{\leq 2} N^\vee \otimes N) = \A^0(\overline{N})$ and choose $v_2$ such that $\dbar v_2 = - R^\bot_2$.
 
\end{proof}

\bibliographystyle{amsalpha}
\addcontentsline{toc}{chapter}{Bibliography}
\bibliography{Todd}

\end{document}